\documentclass[opre,nonblindrev]{informs4}
\usepackage{eqndefns-left} 
\RequirePackage{tgtermes}
\RequirePackage{newtxtext}
\RequirePackage{newtxmath}
\RequirePackage{bm}
\RequirePackage{endnotes}

\OneAndAHalfSpacedXII 

\usepackage{algorithm}
\usepackage{algpseudocode}
\usepackage{tikz}
\usepackage{amsmath}
\usepackage[normalem]{ulem}
\usepackage[
    font=footnotesize,        
    labelfont=normalfont
]{subcaption}

\usetikzlibrary{positioning, arrows.meta,calc,fit,backgrounds}

\usepackage{natbib}
 \bibpunct[, ]{(}{)}{,}{a}{}{,}%
 \def\bibfont{\small}%

\usepackage{enumitem}
\usepackage{subcaption}  %
\usepackage[hidelinks]{hyperref}
\EquationsNumberedThrough    

\TheoremsNumberedThrough     
\ECRepeatTheorems  %

\MANUSCRIPTNO{}

\usepackage{booktabs}
\newcommand{\bx}{\mathbf{x}}

\newcommand{\bz}{\mathbf{z}}

\newcommand{\bQ}{\mathbf{Q}}

\newcommand{\bb}{\mathbf{b}}
\newcommand{\bA}{\mathbf{A}}
\newcommand{\R}{\mathbb{R}}
\newcommand{\ba}{\mathbf{a}}
\newcommand{\bv}{\mathbf{v}}
\newcommand{\bw}{\mathbf{w}}
\newcommand{\bu}{\mathbf{u}}
\newcommand{\by}{\mathbf{y}}

\newcommand{\beee}{\mathbf{e}}

\newcommand{\tpose}{\top}

\newcommand{\btheta}{\boldsymbol{\theta}}
\newcommand{\bTheta}{\boldsymbol{\Theta}}

\newcommand{\bSigma}{\boldsymbol{\Sigma}}
\newcommand{\bepsilon}{\boldsymbol{\epsilon}}
\newcommand{\Xopt}{\mathcal{X}^{\star}}

\newcommand{\bxopt}{\bx^{\star}}
\newcommand{\bthetastar}{\boldsymbol{\theta}^*}
\newcommand{\boldeta}{\boldsymbol{\eta}}

\begin{document}



\RUNAUTHOR{Chan, Sandholtz, and Yousefi}

\RUNTITLE{Uncertainty Quantification in IO}

\TITLE{Uncertainty Quantification in Data-Driven Inverse Optimization via Bayesian Inference}

\ARTICLEAUTHORS{%

\AUTHOR{Timothy C. Y. Chan}
\AFF{Department of Mechanical and Industrial Engineering, University of Toronto}

\AUTHOR{Nathan Sandholtz}
\AFF{Department of Statistics, Brigham Young University}

\AUTHOR{Nasrin Yousefi}
\AFF{Smith School of Business, Queen's University}
}
\ABSTRACT{%
Inverse optimization (IO) is used to estimate unknown parameters of an optimization model from observed decisions. In the data-driven context, the estimated parameters are inherently uncertain,
yet quantifying this uncertainty has received limited attention in the literature, where existing methods return a point estimate. In this paper, we propose a hierarchical Bayesian framework for parameter uncertainty quantification in data-driven inverse optimization. Considering two data-generating processes, we develop two Markov chain Monte Carlo algorithms to estimate the posterior distribution of the unknown parameter vector, which is used to construct credible regions.
We establish posterior consistency under standard identifiability conditions. Numerical experiments demonstrate near-nominal empirical coverage of the credible regions and show that the regions shrink as the number of observed decisions increases.
}%




\KEYWORDS{inverse optimization, Bayesian inference, credible region, Markov chain Monte Carlo,  hierarchical Bayesian model} 
\maketitle


\section{Introduction}
\label{sec:intro}

Inverse optimization (IO) refers to the problem of estimating parameters of an optimization model using decision data. 
These estimates can provide insight into the decision-making process that generated the data \citep{bertsimas2015data,aswani2018inverse, ajayi2022objective,besbes2025contextual, ren2025inverse} or be used in a downstream ``forward'' optimization model that prescribes  decisions \citep{keshavarz2011imputing,chan2014generalized,ronnqvist2017calibrated,zattoni2025inverse}. 
Existing data-driven inverse optimization approaches generally produce a ``point estimate''. 
However, quantifying the uncertainty in these estimates is critical. For example, this quantification can help decision makers calibrate confidence in their inferences or enable them to develop uncertainty sets for downstream robust optimization \citep{lin2024conformal}. 

In this paper, we develop a novel approach to quantify uncertainty in parameters estimated via data-driven inverse optimization, focusing on convex forward problems. 
We propose a fully generative probabilistic model of the observed decision data, explicitly parameterized by the assumed forward problem. 
We then leverage this probabilistic model to infer the parameters via Bayesian inference, yielding both a point estimate and a \textit{credible region} around that estimate. 

Key to our framework is the explicit consideration of a data model. Accordingly, IO-based parameter estimation should incorporate a data model that is consistent with the decision maker's belief about how the data was generated. This paper considers two natural data-generating models. The first assumes that decision data is generated by random perturbations of an optimal solution to an optimization problem. In other words, data noise originates in the space of \textit{decisions}, for example from measurement error. The second assumes the objective function coefficients have been randomly perturbed and the decision maker observes the decisions associated with solving the optimization problem using the perturbed parameters. In this case, noise originates in the space of the \textit{objective function}, for example arising from imperfect knowledge of the true objective function.

By modeling the data-generating process probabilistically, one can view optimization through the lens of statistical inference. In both processes described above, observed decisions are linked to underlying parameters through intermediate, unobserved quantities (either the optimal solution before it is perturbed or the perturbed parameter vector before the forward model is solved). This induces a hierarchical structure in which optimization is an intermediate mapping between parameters and observations. 
Bayesian methods accommodate optimization-generated intermediate variables by conditioning directly on them rather than requiring closed-form marginalization. 

Our specific contributions are as follows. 

\begin{enumerate}
    \item We propose the first methodology for quantifying uncertainty in inverse optimization using a statistically rigorous approach based on Bayesian inference. Our approach inherits classical convergence properties of Bayesian inference. In particular, we show that under standard assumptions the posterior distribution concentrates around the true parameter vector. While our primary focus is uncertainty quantification, the proposed method also yields a new approach to solving data-driven inverse optimization problems. 
    
    \item We develop two Markov chain Monte Carlo (MCMC)-based algorithms to estimate the unknown parameter vector and construct its credible region. The algorithms are tailored to two data-generating processes, requiring the solution of distinct optimization problems. The first solves forward optimization problems, while the second solves inverse optimization problems. 
    
    \item Through extensive numerical experiments on linear and quadratic forward problems, we validate that our algorithms successfully converge to the true parameter vectors. We show that the credible regions shrink as the number of data points increases and that they are well-calibrated to the specified confidence level. 
\end{enumerate}

\section{Literature Review} 
\label{sec:litreview}

The inverse optimization literature can be divided into  classical and data-driven models \citep{chan2025inverse}. Classical models determine parameters of an optimization problem assuming that observed decisions are exactly optimal (i.e., noise free). The models may be linear \citep{ahuja2001inverse}, nonlinear \citep{chow2014nonlinear}, conic \citep{iyengar2005inverse, zhang2010inverse}, or discrete \citep{schaefer2009inverse, wang2009cutting, bulut2021complexity,bodur2022inverse}. Data-driven models assume observed decisions are noisy and thus may not be optimal for any set of parameters of the chosen model \citep{keshavarz2011imputing,bertsimas2015data,esfahani2018data,aswani2018inverse,chan2019inverse,babier2021ensemble,chan2022inverse,chan2022inverse2,gupta2022decomposition,zattoni2025learning}. Our work falls within the data-driven paradigm. 

Statistical concepts have received some attention in the inverse optimization literature. 
\cite{aswani2018inverse} develop an inverse optimization approach that produces statistically consistent estimates for parameters of a convex forward problem. 
Rather than constructing a single consistent estimator, by assuming a fully probabilistic generative model for the data, we obtain a full posterior distribution and prove that it concentrates around the true parameter vector under analogous identifiability conditions.
\cite{shahmoradi2022quantile} show that existing inverse LP methods can lead to unstable estimates with noisy data. They develop an inverse model using a quantile statistic of optimality errors to improve stability of the estimates. The formulations from both papers are NP-hard. Also, in all papers described so far, the models only return a point estimate. 

\cite{birge2022stochastic} study inverse LP where the cost vector is assumed to be drawn from a parametric distribution and the goal is to estimate parameters of this distribution using the observed decisions. Although this approach considers the distributional nature of the parameters, it does not aim to quantify uncertainty in their estimation, which is the focus of our paper. The conformal inverse optimization approach of \cite{chan2024conformal} generates an uncertainty set over the parameters and uses it in a downstream robust optimization problem. However, their approach does not explicitly model the data-generating process. 
In contrast, our approach assumes a fully generative probabilistic model for the decision data, and then leverages Bayesian inference to generate a posterior distribution over the parameter vector, which can be used to construct credible regions at arbitrary credibility levels.

A few studies have applied Bayesian methods in inverse optimization. \cite{sandholtz2023inverse} develop a probabilistic framework for inverse Bayesian optimization to estimate human acquisition functions based on observed decisions in exploration–exploitation tasks. Their method defines a likelihood over behavior driven by a Bayesian optimization subroutine, enabling uncertainty quantification via credible intervals. However, it relies on discretized acquisition families, lacks theoretical consistency guarantees, and is tailored to a specific behavioral setup.
\cite{lu2025bo4io} propose a Bayesian optimization framework for estimating unknown parameters in inverse optimization, treating the inverse problem as a black-box loss minimization task. Uncertainty quantification is addressed via profile likelihood, a classical frequentist technique for constructing confidence intervals.  
However, their paper does not formulate a fully generative probabilistic model for the inverse problem nor does it establish asymptotic guarantees such as posterior consistency.

\section{Statistical Motivation and Background}
\label{sec:statmotiv}

To motivate our approach, it is instructive to contrast loss-based and likelihood-based regression and to review hierarchical Bayesian models.

\subsection{Loss-based vs.\@ Likelihood-based Regression}

Consider a dataset $\mathcal{D} = \{(\bu^1,y^1), (\bu^2,y^2), \ldots, (\bu^N,y^N)\}$. Each $y^i \in \R$ is a realization of a random variable $Y^i$ generated from a mean function $\mu(\bu^i; \btheta) = \mathbb{E}(Y^i\mid\bu^i)$ plus random error, where $\bu^i$ is a vector of covariates and $\btheta \in \mathbb{R}^h$ is a parameter vector.  The regression problem is to estimate $\btheta$ given $\mathcal{D}$. This is typically done via: 1) loss-based minimization, or 2) likelihood-based maximization.

In loss-based regression, a loss function $\ell$ (e.g., squared error loss) penalizes the difference between observed $y^i$ values and the estimated mean function at the corresponding covariate vector $\bu^i$. The parameter vector $\btheta$ is then computed by minimizing the total loss over the dataset:  
$$\hat{\btheta}_\text{LOSS} \in \argmin_{\btheta} \sum_{i=1}^N \ell\left(y^i - \mu(\bu^i; \btheta) \right).$$
Notably, this approach provides no direct uncertainty quantification of the resulting point estimate. 

In likelihood-based regression, a probabilistic model $f$ for the target data is specified (e.g., $Y_i \sim \mathcal{N}(\mu(\bu^i; \btheta), \sigma^2)$). The likelihood function is defined as the joint density of the observed data: 
$$L(\btheta; \mathcal{D}) \;=\; f\!\left(\{y^i\}_{i=1}^N \,\middle|\, \{\bu^i\}_{i=1}^N,\, \btheta \right).$$
Inference may proceed under a Bayesian or frequentist paradigm. For the Bayesian approach, a prior distribution over $\btheta$, $\pi(\btheta)$, is specified and Bayes’ rule is used to obtain the posterior distribution
\begin{equation}\label{eq:posterior}
    \pi(\btheta \mid \mathcal{D}) \propto L(\btheta; \mathcal{D}) \, \pi(\btheta),
\end{equation}
representing the updated belief about $\btheta$ conditional on the observed data $\mathcal{D}$. Point estimates for $\btheta$ are usually based on posterior summaries 
such as the posterior mean, median, or mode.  Uncertainty can be summarized by a \((1 - \alpha)\)-level credible region, 
$\mathcal{C}_{1 - \alpha}$, which satisfies
\begin{equation*}
    \int_{\mathcal{C}_{1 - \alpha}} \pi(\btheta \mid \mathcal{D}) \, d\btheta \geq 1 - \alpha.
\end{equation*}

Relating these paradigms to data-driven inverse optimization, the IO approaches in the literature are almost exclusively loss-based: model parameters $\btheta$ are chosen to minimize some loss between the observed decisions and optimal decisions without considering the distribution of the decision data. The approach we propose is likelihood-based: we treat observed decisions as realizations from the composition of a probabilistic model and an optimization model, defined in Section \ref{sec:modelbasics}.

\subsection{Hierarchical Bayesian Models}

Hierarchical Bayesian models (HBMs) represent complex data-generating processes by decomposing them into conditional layers. As a generic example, consider a model that generates an observation $\by$ from the parameter vector $\btheta$ via the intermediate (latent) variable $\bz$:
\begin{equation}\label{eq:HBM}
\btheta \sim \pi(\cdot), \qquad
\bz \sim p(\cdot \mid \btheta), \qquad
\by \sim f(\cdot \mid \bz).
\end{equation}
The intermediate $\bz$ need not be directly observed and may encode structural relationships in the data-generating process \citep{gelman2007data}. Under this specification, $\by$ is conditionally independent of $\btheta$ given $\bz$ (i.e., $f(\by \mid \bz, \btheta) = f(\by \mid \bz)$), so that the likelihood depends on $\btheta$ only through the intermediate layer. 
This conditional independence structure is a common and convenient modeling choice in hierarchical Bayesian models, as it simplifies the specification of the joint distribution and posterior computation via MCMC \citep{gelman2013bayesian, gelman2007data}.

Hierarchical representations are useful when the mapping from parameters to observations is indirect. Instead of specifying a marginal likelihood $h(\by \mid \btheta)$ in closed form, one can model it through its conditional components. Bayesian inference proceeds by operating on these components, avoiding the need to compute the (possibly intractable) marginal 
$
h(\by \mid \btheta) := \int f(\by \mid \bz)p(\bz \mid \btheta)\, d\bz.
$

Mapping this back to optimization, observed decisions ($\by$) are linked to model parameters ($\btheta$) through the solution of an optimization problem, naturally inducing a hierarchical structure. 

\section{Inverse Optimization: Loss vs.\@ Likelihood} 
\label{sec:modelbasics}

We consider a data-driven IO model where observed decisions are associated with covariates that affect the feasible region and/or objective function of the forward problem. The goal is to estimate a parameter vector in the objective function given covariate-decision pairs.

\subsection{Loss-based Inverse Optimization}
This section reviews the ``standard'' approach to IO. Let the forward problem be
\begin{equation}\label{eq:fo_uy}
\begin{aligned}
        \mathbf{FO}(\bu,\btheta): \quad \underset{\bx}{\operatorname{ minimize}} & \quad g(\bx; \bu, \btheta)\\
         \text{subject to} &\quad\bx \in \mathcal{X}(\bu),
\end{aligned} 
\end{equation}
where $\bx \in \R^n$ is the decision vector, $\bu \in \mathcal{U}$ is a covariate vector, and $\btheta \in \Theta \subseteq \R^h $ is a nonzero, unknown parameter vector to be estimated via inverse optimization. The feasible region $\mathcal{X}(\bu) \subseteq \R^n$ is assumed to be compact, convex and full-dimensional. We assume $g$ is convex in $\bx$ and linear in $\btheta$. 
Since $g$ is linear in $\btheta$, the optimal solution set is invariant under positive rescaling of $\btheta$. Thus, we assume that $\Theta$ is constrained to the unit hypersphere $\mathcal{S}^{h-1} := \{\btheta \in \R^h \;|\; \|\btheta\|_2=1\}$.  
 
Assume the forward problem is solved $N$ times, each time with a different covariate vector $\bu^i, i = 1, \ldots, N$, and the decision maker observes decision data $\by^i, i = 1, \ldots, N$. Specific data-generating processes are described in Section \ref{sec:likelihoodIO}. 
Thus, the dataset available to the inverse optimizer is $\mathcal{D} = \{(\bu^1,\by^1), (\bu^2,\by^2), \ldots, (\bu^N,\by^N)\}$. 
Given $\mathcal{D}$, the data-driven inverse optimization problem is 
\begin{equation}
\label{eq:io_uy}
\begin{aligned}
        \mathbf{IO}(\mathcal{D}): \quad \underset{\btheta}{\operatorname{ minimize}} & \quad \sum_{i=1}^N \ell\left(\by^i, \Xopt(\bu^i,\boldsymbol{\theta})\right)\\
         \text{subject to} &\quad \btheta \in \Theta,
\end{aligned}
\end{equation} 
where
\begin{equation}
    \Xopt(\bu,\btheta) := \underset{\bx \in \mathcal{X}(\bu)}{\arg\min} \; g(\bx; \bu, \btheta)
\end{equation} 
is the set of optimal solutions to $\mathbf{FO}(\bu,\btheta)$ under $\bu$ and $\btheta$, and $\ell$ quantifies the loss between an observed decision $\by^i$ and the optimal solution set $\Xopt(\bu^i,\btheta)$.\endnote{Throughout the paper, we use the superscript $\star$ to denote optimal solutions of an optimization problem (e.g., $\bx^\star(\bu,\btheta)$), and the superscript $*$ to denote true but unknown parameter values (e.g., $\bthetastar$).} The two most common choices for the loss function are the Euclidean distance between $\by^i$ and the optimal solution set (\textit{decision loss}), 
\begin{equation*}
\ell_d\left(\by^i, \Xopt(\bu^i,\boldsymbol{\theta})\right) := \min_{\bz^i \in \Xopt(\bu^i,\btheta)}\|\by^i - \bz^i\|_2
\end{equation*}
and the difference between the objective value of $\by^i$ and the optimal value (\textit{suboptimality loss}),
\begin{equation*}
\ell_o\left(\by^i, \Xopt(\bu^i,\boldsymbol{\theta})\right) := g(\by^i; \bu^i,\btheta) - 
g(\bz^i; \bu^i,\btheta)
\end{equation*}
 for any $\bz^i \in \Xopt(\bu^i,\btheta)$. 

Solving \eqref{eq:io_uy} yields a point estimate
\begin{equation}
\label{eq:io_uy_solve}
        \hat{\boldsymbol{\theta}}_{\text{LOSS}} \in \underset{\btheta \in \bTheta}{\arg \min} \sum_{i=1}^N \ell\left(\by^i, \Xopt(\bu^i,\boldsymbol{\theta})\right).  
\end{equation} 
This \textit{loss-based} approach does not  consider uncertainty around $\hat{\boldsymbol{\theta}}_{\text{LOSS}}$.  
Also, different loss functions will yield different estimates, 
emphasizing the need to provide uncertainty quantification.

\subsection{Likelihood-based Inverse Optimization}
\label{sec:likelihoodIO}

Here, we introduce a \emph{likelihood-based} approach to IO. 
We statistically model the data-generating process that produces the observed decisions. 
This induces a probability distribution over the observations that we use to construct a set of likely values for $\btheta$.
Next, we describe two data-generating processes (also referred to as error processes) that align with decision loss and suboptimality loss.

\subsubsection{Decision-Space Error.}
\label{sec:decerror}
This error process assumes perturbations are applied directly to optimal solutions of the forward problem, so that the inverse optimizer observes noisy realizations of these solutions. Let $\bxopt(\bu^i,\btheta^*)$ be an optimal solution to $\mathbf{FO}(\bu,\btheta^*)$.  
Conditional on $\bxopt(\bu^i,\btheta^*)$, and for a given probability distribution $f$ on $\mathbb{R}^n$, the decision maker observes
\begin{equation}
\label{eq:dec_error}
\by^i \sim f(\cdot \mid \bxopt(\bu^i,\btheta^*)).
\end{equation}
Figure \ref{fig:graphical_dec_error} shows a graphical model depicting this data-generating process.

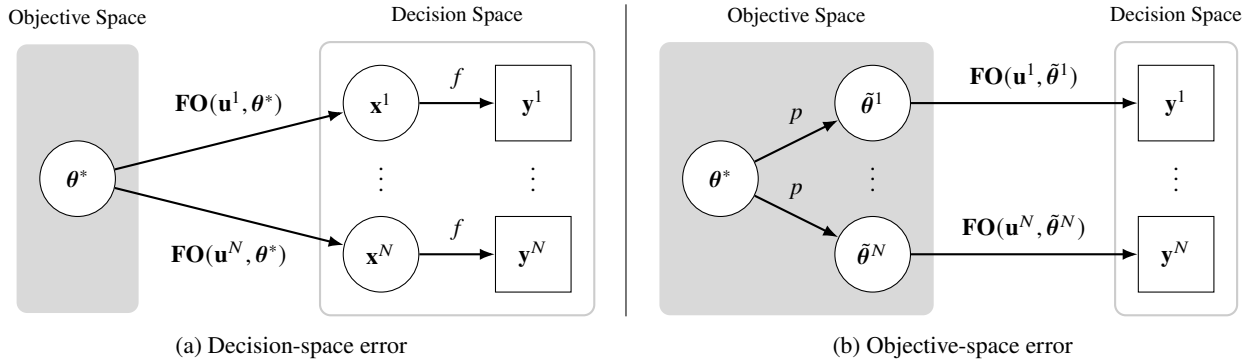
\begin{figure}[H]
\begin{subfigure}[t]{0.47\textwidth}
\centering
\begin{tikzpicture}[
  font=\footnotesize,
  latent/.style={draw,circle,minimum size=1cm,fill=white},
  obs/.style={draw,minimum size=1cm,fill=white},
  arrow/.style={-{Latex[length=2mm]},thick},
  panel/.style={fill=gray!30,rounded corners,inner sep=0.3cm,draw=none}  
]

  \node[latent] (theta) {$\bthetastar$};


  \node[latent,right=3cm of theta, yshift=1cm] (x1) {$\bx^1$};
  \node[latent,right=3cm of theta, yshift=-1cm] (xN) {$\bx^N$};

  \node[obs,right=1cm of x1] (y1) {$\by^1$};
  \node[obs,right=1cm of xN] (yN) {$\by^N$};

  \begin{pgfonlayer}{background}

    \node[
      panel,
      fill=none,
      draw=gray!40,
      thick,
      fit=(x1)(xN)(y1)(yN)
    ] (decspanel) {};
  \node[
    panel,
    inner ysep=1.27cm,  
    fit=(theta)
  ] (objpanel) {};
  \end{pgfonlayer}

  \node[font=\scriptsize, anchor=north] at ($(objpanel.north)+(0,0.6cm)$) {Objective Space};
  \node[font=\scriptsize, anchor=north] at ($(decspanel.north)+(0,0.6cm)$) {Decision Space};

  \node at ($(x1)!0.45!(xN)$)  {$\vdots$};
  \node at ($(y1)!0.45!(yN)$)  {$\vdots$};


  \draw[arrow] (theta) --
  node[midway, above=5pt]
  {$\mathbf{FO}(\bu^1,\btheta^*)$}
  (x1);
  \draw[arrow] (theta) --  
  node[midway, below=5pt]
  {$\mathbf{FO}(\bu^N,\btheta^*)$}
  (xN);

  \draw[arrow] (x1) -- node[midway, above=1pt] {$f$} (y1);
  \draw[arrow] (xN) -- node[midway, above=1pt] {$f$} (yN);

\end{tikzpicture}
\caption{Decision-space error}
\label{fig:graphical_dec_error}
\end{subfigure}
\hfill
\rule{0.4pt}{0.25\textwidth}
\hfill
\begin{subfigure}[t]{0.47\textwidth}
\centering
\begin{tikzpicture}[
  font=\footnotesize,
  latent/.style={draw,circle,minimum size=1cm,fill=white},
  obs/.style={draw,minimum size=1cm,fill=white},
  arrow/.style={-{Latex[length=2mm]},thick},
  panel/.style={fill=gray!30,rounded corners,inner sep=0.3cm,draw=none}  
]

  \node[latent] (theta) {$\bthetastar$};

  \node[latent,right=1cm of theta, yshift=1cm] (tt1) {$\tilde{\btheta}^1$};
  \node[latent,right=1cm of theta, yshift=-1cm] (ttN) {$\tilde{\btheta}^N$};

  \node[obs,right=3cm of tt1] (y1) {$\by^1$};
  \node[obs,right=3cm of ttN] (yN) {$\by^N$};


  \begin{pgfonlayer}{background}
    \node[panel,fit=(theta)(tt1)(ttN)] (objpanel) {};

    \node[
      panel,
      fill=none,
      draw=gray!40,
      thick,
      fit=(y1)(yN)
    ] (decspanel) {};
  \end{pgfonlayer}

  \node[font=\scriptsize, anchor=north] at ($(objpanel.north)+(0,0.6cm)$) {Objective Space};
  \node[font=\scriptsize, anchor=north] at ($(decspanel.north)+(0,0.6cm)$) {Decision Space};

  \node at ($(tt1)!0.45!(ttN)$) {$\vdots$};
  \node at ($(y1)!0.45!(yN)$)  {$\vdots$};

  \draw[arrow] (theta) -- node[midway, above=2pt] {$p$}(tt1);
  \draw[arrow] (theta) -- node[midway, above=2pt]{$p$} (ttN);

  \draw[arrow] (tt1) -- node[midway, above=1pt] 
  {$\mathbf{FO}(\bu^1,\tilde{\btheta}^1)$}
  (y1);
   \draw[arrow] (ttN) -- node[midway, above=1pt] 
  {$\mathbf{FO}(\bu^N,\tilde{\btheta}^N)$}
  (yN);


\end{tikzpicture}
  \caption{Objective-space error}  \label{fig:graphical_obj_error}
\end{subfigure}
\caption{Graphical models illustrating the data-generating process under decision-space (a) and objective-space error (b).}

\label{fig:graphical_model}
\end{figure}

\begin{example}\label{example:decspace}
Figure~\ref{fig:dec_tikz_and_data} illustrates two-dimensional minimization problems with a linear objective, under both linear constraints and quadratic constraints. Three instances are depicted in each panel, each corresponding to a different $\bu^i$. Instance $i$ has a unique optimal solution $\bx^i := \bxopt(\bu^i,\bthetastar)$. The decision maker observes $\by^i$, where $\by^i$ equals $\bx^i$ plus Gaussian noise. Note that the observations can be feasible or infeasible solutions.

\end{example}

\begin{figure}[!htbp]
  \centering

  \begin{subfigure}{0.32\linewidth}
    \centering
    \includegraphics[width=\linewidth,
        trim=0in 0in 0in .07in,clip]{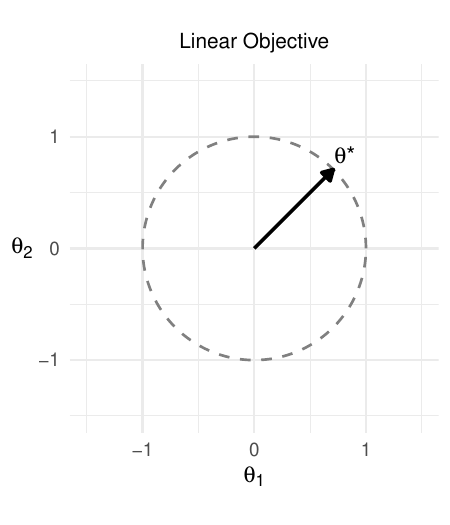}
    \caption{Objective space.}
  \end{subfigure}
\hspace{1pt}
{\color{gray}\raisebox{-0.085\linewidth}{\rule{1pt}{0.3\linewidth}}}
\hspace{1pt}  
\begin{subfigure}{0.64\linewidth}
    \centering
    \includegraphics[width=\linewidth,
        trim=.1in .05in .02in .02in,clip]{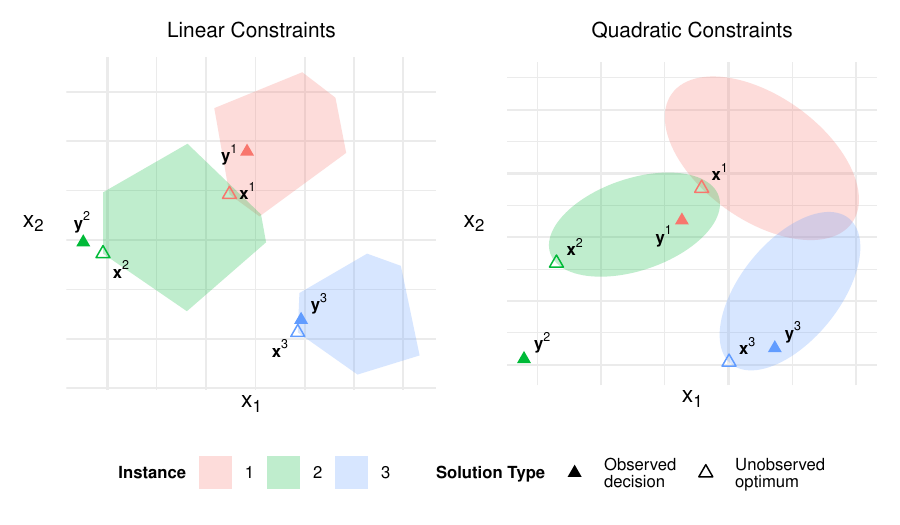}
    \caption{Decision space.}
  \end{subfigure}
\vspace{10pt}
 \caption{(a) Unobserved parameter vector $\btheta^*$. (b) Unobserved optimal solutions and their observed perturbations.}
 
  \label{fig:dec_tikz_and_data}
\end{figure}

\subsubsection{Objective-Space Error.}
\label{sec:objerror}

This error process assumes that the parameter vector $\bthetastar$ is perturbed \textit{before} the forward problem is solved. The resulting optimal solution is what the inverse optimizer observes. Let $p$ be a probability distribution over the unit hypersphere $\mathcal S^{h-1}$. Then, 
\begin{equation}
\label{eq:objective_error_data}
\by^i = \bxopt(\bu^i,\tilde{\btheta}^i) \in
\Xopt(\bu^i,\tilde{\btheta}^i),
\end{equation}
where 
$\tilde{\btheta}^i \sim p(\cdot \mid \bthetastar)$.
Figure \ref{fig:graphical_obj_error} shows a graphical model depicting this data-generating process.

\begin{example}\label{example:objspace}
The setup is the same as in Example \ref{example:decspace}. 
The left panel in Figure~\ref{fig:obj_tikz_and_data} illustrates the perturbed parameter vectors $\tilde{\btheta}^1$, $\tilde{\btheta}^2$, $\tilde{\btheta}^3$, which are generated from a von Mises-Fisher distribution (see~\ref{ec:vmf}) centered at $\bthetastar$. The right panel shows the optimal solutions corresponding to $\tilde{\btheta}^i$ for $i = 1, 2, 3.$ 
In contrast to Example \ref{example:decspace}, all observations lie on the boundary of the feasible region.
\end{example}

In both error processes, when $\Xopt (\bu, \btheta)$ is not a singleton, let $\bx^\star(\bu, \btheta)$ be chosen from this set under some rule independent of $\btheta^*$; we discuss implications of non-uniqueness in Section~\ref{sec:identifiability}.

\begin{figure}[!htbp]
   \centering

  \begin{subfigure}{0.32\linewidth}
    \centering
    \includegraphics[width=\linewidth,
        trim=0in 0in 0in .07in,clip]{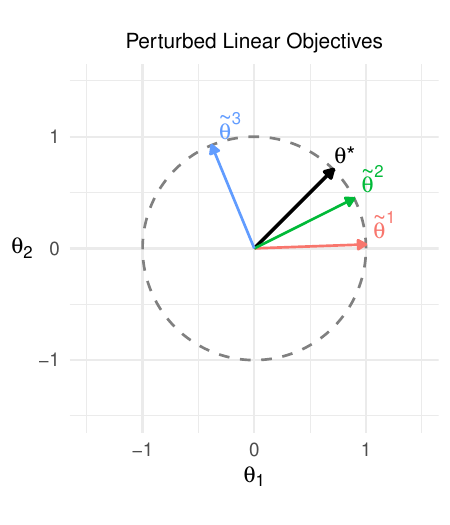}
    \caption{Objective space.}
  \end{subfigure}
\hspace{1pt}
{\color{gray}\raisebox{-0.085\linewidth}{\rule{1pt}{0.3\linewidth}}}
\hspace{1pt}  
\begin{subfigure}{0.64\linewidth}
    \centering
    \includegraphics[width=\linewidth,
        trim=.1in .05in .02in .02in,clip]{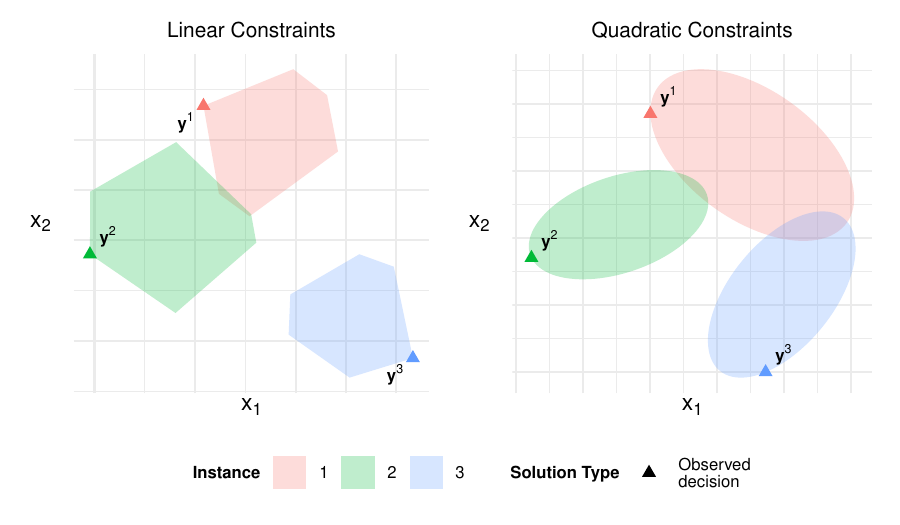}
    \caption{Decision space.}
  \end{subfigure}
\vspace{10pt}
 \caption{(a) Unobserved parameter vector $\btheta^*$ and its perturbations. (b) Observed optimal solutions due to perturbed cost vectors.}
  \label{fig:obj_tikz_and_data}
\end{figure}

\section{Inverse Optimization via Bayesian Inference}
\label{sec:bayesianIO}

In this section, we develop a new framework for data-driven IO using hierarchical Bayesian methods, inspired by the likelihood-based regression approach outlined in Section~\ref{sec:statmotiv}. 

\subsection{Inverse Optimization as a Hierarchical Bayesian  Model}
\label{sec:hier_model}

Both data-generating processes from Section \ref{sec:likelihoodIO} can be viewed as a hierarchical Bayesian model, with the structure of its intermediate layer dependent on the error model. In decision-space error, the intermediate layer is the latent optimal solution of the forward problem. In objective-space error, it is the randomly perturbed parameter vector. 
These hierarchies are shown in Figure~\ref{fig:graphical_model}.

\subsection{Prior Specification}
Given the likelihood specifications introduced in Section 4, Bayesian inference proceeds by placing a prior distribution on the parameter vector $\btheta$ and combining it with the hierarchical structure above.
Given the assumed scale-invariance property, we restrict the support of the prior distribution $\pi(\cdot)$  to the unit hypersphere $\mathcal{S}^{h-1} := \{ \btheta \in \mathbb{R}^h \mid \|\btheta\|_2 = 1 \}$. The uniform distribution over $\mathcal{S}^{h-1}$ is the natural choice for a non-informative prior.  If prior information about the direction of $\btheta^*$ is available, a more informative prior may be used, such as the von Mises–Fisher (vMF) distribution.

\subsection{Posterior Approximation}
\label{sec:mcmc}

The posterior distribution is not available in closed form due to the likelihood specification. Thus, we approximate it using samples from a Metropolis–Hastings algorithm \citep{metropolis1953equation, hastings1970monte}. The implementation depends on the assumed data-generating process. 

\subsubsection{Decision-Space Error.}

Under decision-space error, the likelihood $f(\by^i \mid \bxopt(\bu^i, \btheta^*))$ depends on $\btheta^*$ through the solution to the forward problem $\mathbf{FO}(\bu^i, \btheta^*)$. Each iteration therefore requires solving $\mathbf{FO}(\bu^i, \btheta^*)$ to evaluate the likelihood at the proposed parameter vector.

Algorithm \ref{algo:decision} describes the approach. It begins by drawing $\btheta^0$ from a prior distribution $\pi(\cdot)$ (Step 2). Then it solves the forward problem $\mathbf{FO}(\bu^i,\btheta^0)$ for each $\bu^i$ to obtain optimal solutions $\bx_1^0, \ldots, \bx_N^0$ (Step 3), and computes the corresponding log likelihood (Step 4). In iteration $t$, a candidate $\btheta^c$ is generated from a \textit{proposal distribution} $q(\cdot \mid \btheta^{t-1})$ with support on $\mathcal{S}^{h-1}$. The forward problem is solved with $\btheta^c$ for each instance (Step 7) and the log likelihood evaluated (Step 8). The candidate is accepted with probability $\alpha$; otherwise the Markov chain remains at $\btheta^{t-1}$ (Steps 9--13). 
Assuming that both $q(\cdot \mid \btheta^{t-1})$ and $\pi(\btheta)$ are positive on all of $\mathcal{S}^{h-1}$, the resulting Markov chain is irreducible and targets the posterior distribution $\pi(\btheta \mid \mathcal{D})$ under standard Metropolis--Hastings conditions \citep{hastings1970monte}.
When the prior $\pi(\btheta)$ is uniform and the proposal $q$ is symmetric, the prior and proposal ratios in the acceptance probability equal 1. 

To assess convergence to the stationary distribution, we run several chains via Algorithm~\ref{algo:decision} and use the potential scale reduction factor (PSRF) \citep{gelman1992inference} to choose a sufficiently long warm-up period. We discard the initial $T_{\text{warm}}$ iterations and use remaining samples for inference.

In practice, other parameters governing the likelihood $f$ (e.g., the covariance matrix in a Gaussian likelihood) are generally unknown and must be estimated alongside $\btheta$. This is straightforward to do by treating these parameters as additional unknowns and sampling them jointly with $\btheta$ in the same loop (Steps 5--13). We omit this from Algorithm~\ref{algo:decision} for brevity, as the extension is standard.  
Implementation details are given for the specific numerical experiments we perform in Section \ref{sec:numericals}.

\begin{algorithm}[htbp]
\caption{Metropolis--Hastings Sampling for Decision-Space Error}
\label{algo:decision}

\begin{algorithmic}[1]
\State \textbf{Input:} dataset $\mathcal{D} = \{(\bu^1,\by^1), \ldots, (\bu^N,\by^N)\}$, forward problem $\mathbf{FO}(\bu,\btheta)$, prior distribution $\pi(\cdot)$, likelihood $f(\by \mid \bxopt)$,
proposal distribution $q(\cdot \mid \btheta)$,
warm-up period $T_{\text{warm}}$, post warm-up iterations $T$
\State Draw $\btheta^0 \sim \pi(\cdot)$
\State Solve $\mathbf{FO}(\bu^i,\btheta^0)$ to obtain $\bxopt(\bu^i,\btheta^0)$ for all $i = 1, \ldots, N$
\State Compute log-likelihood $L_f(\btheta^0;\mathcal{D}):=\sum_{i=1}^N \log f\!\left(\by^i \mid \bxopt(\bu^i,\btheta^0)\right)$
\For{$t = 1,\ldots, T_\text{warm}+T$}
    \Statex \textbf{Draw candidate:} 
    \State Draw $\btheta^c \sim q(\cdot \mid \btheta^{t-1})$
    \State Solve $\mathbf{FO}(\bu^i,\btheta^c)$ to obtain $\bxopt(\bu^i,\btheta^c)$ for all $i = 1, \ldots, N$
\State Compute log-likelihood $L_f(\btheta^c;\mathcal{D}):=\sum_{i=1}^N \log f\!\left(\by^i \mid \bxopt(\bu^i,\btheta^c)\right)$  
    \Statex \textbf{Accept or reject:}
\State Compute acceptance probability 
$\alpha \gets \min\!\left\{1,\;\exp\!\left(L_f(\btheta^c;\mathcal D)-L_f(\btheta^{t-1};\mathcal D)\right) \cdot \dfrac{\pi(\btheta^c)\, q(\btheta^{t-1} \mid \btheta^c)}{\pi(\btheta^{t-1})\, q(\btheta^c \mid \btheta^{t-1})}\right\}$
\State Draw $r \sim \text{Uniform}(0,1)$
\State $\btheta^t \gets 
\begin{cases}
\btheta^c & \text{if } r \le \alpha, \\
\btheta^{t-1} & \text{otherwise.}
\end{cases}$
\EndFor
\State \Return $\{\btheta^{T_{\text{warm}} + 1},\ldots,\btheta^{T_\text{warm}+T}\}$
\end{algorithmic}
\end{algorithm}

\subsubsection{Objective-Space Error.}

In contrast to the decision-space error model, in the objective-space error model the perturbation occurs \emph{before} the forward optimization model is solved, and the observed decisions are exact solutions to $\mathbf{FO}(\bu^i, \tilde{\btheta}^i)$. This structure suggests two-step approach: first recover the perturbed parameter vectors $\tilde{\btheta}^1, \ldots, \tilde{\btheta}^N$ from the observed decisions via inverse optimization, and then use $p(\tilde{\btheta}^i \mid \cdot)$ as the basis for inference.

Algorithm~\ref{algo:objective} implements this approach. 
In Step 2, an inverse optimization problem is solved for each observation $(\bu^i, \by^i)$: 
\begin{equation}    
\label{eq:inverse_prob_alg2}
    \tilde{\btheta}^{i} \in 
\underset{\btheta \in \Theta}{\arg\min} \; \mathbf{IO}(\bu^i,\by^i),
\quad i = 1,\ldots,N.
\end{equation}
Formulation \eqref{eq:inverse_prob_alg2} specializes \eqref{eq:io_uy} to the case of a single observed decision. 
In the case of objective-space error, $\by^i \in \mathcal{X}^\star(\bu^i, \tilde{\btheta}^i)$ for some $\tilde{\btheta}^i \in \Theta$, so the optimal value of \eqref{eq:inverse_prob_alg2} is zero for any choice of loss function $\ell$. When the minimizer is unique, $\ell$ does not affect $\tilde{\btheta}^i$. When it is not, different choices of $\ell$ may select different solutions, which we discuss further below.
The resulting $\tilde{\bTheta} = \{\tilde{\btheta}^1, \ldots, \tilde{\btheta}^N\}$ is fixed throughout the subsequent MCMC steps. In Steps~4 and~7, the perturbation distribution $p(\tilde{\btheta}^i \mid \btheta)$ replaces the likelihood $f(\by^i \mid \bxopt(\bu^i, \btheta))$ as the basis for the evaluations. The remainder of the algorithm is essentially unchanged, including determining the warm-up period using PSRF.

Figure~\ref{fig:posterior} shows histograms of 10{,}000 MCMC samples (post-warm-up) of the angular coordinate of $\btheta$ drawn from Algorithms~\ref{algo:decision} and~\ref{algo:objective} applied to the data in Examples~1 and 2, respectively.  

\begin{algorithm}[htbp]
\caption{Metropolis--Hastings Sampling for Objective-Space Error}
\label{algo:objective}
\begin{algorithmic}[1]
\State \textbf{Input:} dataset $\mathcal{D} = \{(\bu^1,\by^1), \ldots, (\bu^N,\by^N)\}$, prior distribution $\pi(\cdot)$, perturbation distribution $p(\tilde{\btheta} \mid \btheta)$, 
proposal distribution $q(\cdot \mid \btheta)$,
warm-up period $T_{\text{warm}}$, post warm-up iterations $T$ 
\State Solve $\mathbf{IO}(\bu^i,\by^i)$ for all $i = 1,\ldots,N$ to obtain $\tilde{\bTheta}=\{\tilde{\btheta}^1,\ldots,\tilde{\btheta}^N\}$
\State Draw $\btheta^0 \sim \pi(\cdot)$
\State Compute $L_p(\btheta^0; \tilde{\bTheta}) := \sum_{i=1}^N \log p(\tilde{\btheta}^i \mid \btheta^0)$
\For{$t=1,\ldots,T_\text{warm}+T$}
    \Statex \textbf{Draw candidate:}
    \State Draw $\btheta^c \sim q(\cdot \mid \btheta^{t-1})$
    \State Compute $L_p(\btheta^c; \tilde{\bTheta}) := \sum_{i=1}^N \log p(\tilde{\btheta}^i \mid \btheta^c)$
    \Statex \textbf{Accept or reject:}
    \State Compute acceptance probability $\alpha \gets \min\!\left\{1,\; \exp\!\left(L_p(\btheta^c;\tilde{\bTheta}) - L_p(\btheta^{t-1};\tilde{\bTheta})\right) \cdot \dfrac{\pi(\btheta^c)\, q(\btheta^{t-1} \mid \btheta^c)}{\pi(\btheta^{t-1})\, q(\btheta^c \mid \btheta^{t-1})}\right\}$
    \State Draw $r \sim \text{Uniform}(0,1)$
\State $\btheta^t \gets 
\begin{cases}
\btheta^c & \text{if } r \le \alpha, \\
\btheta^{t-1} & \text{otherwise.}
\end{cases}$
\EndFor
\State \Return $\{\btheta^{T_{\text{warm}} + 1},\ldots,\btheta^{T_\text{warm}+T}\}$
\end{algorithmic}
\end{algorithm}

\begin{remark} [Non-uniqueness of \eqref{eq:inverse_prob_alg2}]
When an optimal solution $\tilde{\btheta}^i$ to \eqref{eq:inverse_prob_alg2} is not unique, it lies within a set of observationally equivalent parameters, and the specific representative selected depends on the choice of loss function $\ell$. In this case, Algorithm~\ref{algo:objective} targets a plug-in approximation $\pi(\btheta \mid \{\tilde{\btheta}^i\}_{i=1}^N)$, based on a representative parameter vector recovered from each observation, rather than the full posterior $\pi(\btheta \mid \mathcal{D})$. However, the impact of this distinction diminishes asymptotically as new information accumulates through covariate variation, which we discuss in detail in Section~\ref{sec:identifiability}.
\end{remark}

\begin{remark}[Algorithm comparison]
Regarding optimization, Algorithm \ref{algo:objective} has two key differences from Algorithm~\ref{algo:decision}. First, no optimization problems are solved inside the MCMC loop (Steps 5--11). Second, solving the inverse problem in Step~2 requires identifying a parameter vector that rationalizes a single observation, which can be done in closed form when the objective is linear in $\btheta$ and the observation is on the boundary of the feasible region. 
Importantly, these differences exist because the algorithms are specialized to two different data-generating processes. The choice of algorithm should be guided by the observed data and the hypothesized data-generating mechanism. 
\end{remark}

\subsection{Identifiability and Consistency}
\label{sec:identifiability}

Here we characterize the conditions under which the parameter vector $\btheta^*$ is identifiable from the observed data and the posterior distribution is consistent as $N \to \infty$.

\begin{definition}[Identifiability]
The parameter $\btheta^* \in \Theta$ is said to be \emph{identifiable} if
\begin{equation*}
h(\cdot \mid \btheta) = h(\cdot \mid \btheta^*)
\Longrightarrow
\btheta = \btheta^*,
\end{equation*}
where $h(\by \mid \btheta)$ is the data-generating distribution of the observations $\by$ under parameter $\btheta$.
\end{definition}

In standard parametric models, identifiability is determined by the likelihood, and posterior consistency follows under mild regularity conditions on the model and prior \citep{van2000asymptotic}. In inverse optimization, the mapping from $\btheta$ to observed data is mediated by an optimization problem, and whether distinct parameters induce distinct data-generating distributions depends on the structure of the forward problem; e.g., when the mapping $\btheta \mapsto \bx^\star(\bu, \btheta)$ is one-to-one, classical results apply directly.

For each $\bu \in \mathcal{U}$, define the \emph{instance-specific equivalence class}
\begin{equation*}
\Theta^*(\bu) := \left\{ \btheta \in \Theta \;\middle|\; \bx^\star(\bu,\btheta) = \bx^\star(\bu,\btheta^*) \right\},   
\end{equation*}
which contains all parameter vectors that generate the same optimal solution as $\btheta^*$ for covariate $\bu$. When $\Theta^*(\bu)$ is not a singleton, the covariate $\bu$ alone cannot distinguish $\btheta^*$ from other elements of $\Theta^*(\bu)$, regardless of the number of observations at that covariate value. 

Across multiple observations with distinct covariates, however, identifying information on $\btheta^*$  accumulates when the covariates induce different optimal solutions for different parameter vectors.  Given a dataset $\mathcal{D} = \{(\bu^i, \by^i)\}_{i=1}^N$, define the \emph{globally identified set}
\[
\Theta^*_N := \bigcap_{i=1}^N \Theta^*(\bu^i).
\]
Even if each $\Theta^*(\bu^i)$ is large, $\Theta^*_N$ may shrink rapidly as $N$ increases. 
Lemma \ref{lem:set_collapse} below formalizes conditions under which the set $\Theta^*_N$ converges to the singleton $\{\btheta^*\}$. 



\begin{lemma}[Identification]
\label{lem:set_collapse}
Let $\bu^i \sim \mathbb{P}_{\mathcal{U}}$ be i.i.d.\ and $\Theta$ be compact. Assume that for every $\btheta \neq \btheta^*$, there exists an open neighborhood $\mathcal{O}_{\btheta}$ of $\btheta$ and a measurable set $\mathcal{B}_{\btheta} \subseteq \mathcal{U}$ with $\mathbb{P}_{\mathcal{U}}(\mathcal{B}_{\btheta}) > 0$ such that $\bx^\star(\bu, \btheta') \neq \bx^\star(\bu, \btheta^*)$ for all $\btheta' \in \mathcal{O}_{\btheta}$ and $\bu \in \mathcal{B}_{\btheta}$. 
Then $\Theta_N^* \to \{\btheta^*\}$ almost surely as $N \to \infty$.
\end{lemma}

\begin{proof} {Proof.}
Let $(\mathcal{V}_k)_{k=1}^\infty$ be a decreasing sequence (i.e., $\mathcal{V}_{k+1} \subseteq \mathcal{V}_k$) of open neighborhoods of $\btheta^*$ such that $\bigcap_{k=1}^\infty \mathcal{V}_k = \{\btheta^*\}$. Fix $k$. Since $\Theta$ is compact and $\mathcal{V}_k$ is open, the set $\Theta \setminus \mathcal{V}_k$ is compact. The collection $\{\mathcal{O}_{\btheta} \mid \btheta \in \Theta \setminus \mathcal{V}_k\}$ is an open cover of $\Theta \setminus \mathcal{V}_k$, so by compactness there exists a finite subcover $\{\mathcal{O}_{\btheta_1}, \ldots, \mathcal{O}_{\btheta_m}\}$ with corresponding covariate sets $\mathcal{B}_{\btheta_1}, \ldots, \mathcal{B}_{\btheta_m}$.

For each $j \in \{1, \ldots, m\}$, since $\mathbb{P}_{\mathcal{U}}(\mathcal{B}_{\btheta_j}) > 0$ and the $\bu^i$ are i.i.d., the Borel-Cantelli lemma implies that, almost surely, there is some $i_j$ such that $\bu^{i_j} \in \mathcal{B}_{\btheta_j}$. For this $\bu^{i_j}$, there is no $\btheta' \in \mathcal{O}_{\btheta_j}$ that lies in $\Theta^*(\bu^{i_j})$. Since $\Theta_{N+1}^* \subseteq \Theta_N^*$, it follows that $\Theta_{N}^* \cap \mathcal{O}_{\btheta_j} = \emptyset$ for all $N \ge i_j$.

Therefore, for all $N \ge \max_{1\le j\le m} i_j$,  $\Theta_{N}^* \subseteq \Theta \setminus \cup_{j=1}^m\mathcal{O}_{\btheta_j} \subseteq \mathcal{V}_k.$ Thus, for this fixed $k$, $\Theta_{N}^* \subseteq \mathcal{V}_k$ almost surely for all sufficiently large $N$. 
Taking a countable intersection over $k$, we obtain that almost surely, for every $k$, $\Theta_N^* \subseteq \mathcal{V}_k$ eventually. Since $\bigcap_k \mathcal{V}_k = \{\btheta^*\}$, $\Theta_N^* \to \{\btheta^*\}$ almost surely.\Halmos
\end{proof}

Lemma \ref{lem:set_collapse} shows that covariate variation yields asymptotic identification of $\btheta^*$. 
The next result establishes that under suitable conditions, this identification is reflected in the posterior distribution as well. 
Before stating the result, note that both error processes from Section \ref{sec:likelihoodIO} can be expressed using a common representation for the observation density conditional on the optimal solution given $\btheta$, denoted $\nu(\by \mid \bxopt(\bu, \btheta))$. This provides a convenient abstraction for the analysis below. For decision-space error, set $\nu = f$. For objective-space error, 
\begin{equation}
\label{eq:obj_err_likelihood}
    \nu(\by \mid \bx^\star(\bu, \btheta)) = \int_{\left\{ \tilde{\btheta} \in \Theta \;\middle|\; 
\bx^\star(\bu,\tilde{\btheta}) = \by \right\}} 
p(\tilde{\btheta} \mid \btheta) \, d\tilde{\btheta}.
\end{equation}
This likelihood is determined by a unique perturbed parameter vector when each $\tilde{\btheta}$ generates a distinct $\bx^\star(\bu, \tilde{\btheta})$.

\begin{theorem}[Posterior consistency]
\label{th:identification}
Let \(\mathcal{D}_N = \{(\bu^i, \by^i)\}_{i=1}^N\) be i.i.d. samples, where  \(\bu^i \sim \mathbb{P}_{\mathcal{U}}\) and \(\by^i \mid \bu^i \sim \nu(\cdot \mid \bx^\star(\bu^i, \btheta^*))\).  Let $\pi(\btheta)$ be the prior distribution, $\nu(\cdot \mid \bx^\star(\bu, \btheta))$ be the likelihood function, and $\pi(\btheta \mid \mathcal{D}_N)$ be the posterior obtained via Bayesian updating according to \eqref{eq:posterior}.
Let  $D\!\left(\cdot\,\middle\|\,\cdot\right)$ be the Kullback–Leibler (KL) divergence and $\mathcal{U}(\btheta) := \left\{ \bu \in \mathcal{U} \mid \bx^\star(\bu,\btheta) \neq \bx^\star(\bu,\btheta^*) \right\}$. Assume: 
\begin{enumerate} [label=(A\arabic*)]

    \item \textbf{(Covariate variation)}  For all $\btheta \neq \btheta^*$, $\mathbb{P}_{\mathcal{U}}(\mathcal{U}(\btheta)) > 0$.

    \item \textbf{(Observation model)} For all $\bx \neq \bx'$, the conditional densities satisfy $\nu(\cdot \mid \bx) \neq \nu(\cdot \mid \bx')$ and are defined with respect to the same underlying measure.

    \item \textbf{(Regularity at $\btheta^*$)} For all $\bu$ (except possibly on a $\mathbb{P}_{\mathcal{U}}$-measure zero set), $\bx^\star(\bu, \btheta)$ is continuous in $\btheta$ at $\btheta^*$ and $\nu(\cdot \mid \bx)$ is continuous in $\bx$. Furthermore, there exists a $\mathbb{P}_{\mathcal{U}}$-integrable function $\zeta: \mathcal{U} \to \mathbb{R}_+$ and a neighborhood of $\btheta^*$ on which $D\!\left(\nu(\cdot \mid \bx^\star(\bu, \btheta^*)) \,\middle\|\, \nu(\cdot \mid \bx^\star(\bu, \btheta))\right) \leq \zeta(\bu)$.

    \item \textbf{(Prior support)} The prior $\pi(\btheta)$ assigns positive mass to every neighborhood of $\btheta^*$.

\end{enumerate}
Then the posterior satisfies $\pi(\btheta \in \mathcal{A} \mid \mathcal{D}_N) \to 1$ almost surely for every open set $\mathcal{A}$ containing $\btheta^*$. 

\end{theorem}

\begin{proof}{Proof.}

Let
    $K(\btheta^*, \btheta) := \int_{\mathcal{U}} D\!\left(\nu(\cdot \mid \bx^\star(\bu,\btheta^*)) \,\middle\|\, \nu(\cdot \mid \bx^\star(\bu,\btheta))\right) d\mathbb{P}_{\mathcal{U}}$.
For any $\btheta \neq \btheta^*$ and any $\bu \in \mathcal{U}(\btheta)$, $\nu(\cdot \mid \bx^\star(\bu,\btheta)) \ne \nu(\cdot \mid \bx^\star(\bu,\btheta^*))$ by (A2), so $D\!\left(\nu(\cdot \mid \bx^\star(\bu,\btheta^*)) \,\middle\|\, \nu(\cdot \mid \bx^\star(\bu,\btheta))\right) > 0$. 
Since $\mathcal{U}(\btheta)$ has positive $\mathbb{P}_{\mathcal{U}}$-measure by (A1), it follows that $K(\btheta^*, \btheta) > 0$ for all $\btheta \neq \btheta^*$. 

Next we show that $K(\btheta^*, \cdot)$ is continuous at $\btheta^*$. By (A3), for almost every $\bu$, $\bx^\star(\bu, \btheta) \to \bx^\star(\bu, \btheta^*)$ as $\btheta \to \btheta^*$, and so $\nu(\cdot \mid \bx^\star(\bu, \btheta)) \to \nu(\cdot \mid \bx^\star(\bu, \btheta^*))$. 
Hence, 
$D\!\left(\nu(\cdot \mid \bx^\star(\bu, \btheta^*)) \,\middle\|\, \nu(\cdot \mid \bx^\star(\bu, \btheta))\right) \to 0$ pointwise as $\btheta \rightarrow \btheta^*$. Combined with the integrability bound in (A3), the dominated convergence theorem implies $K(\btheta^*, \btheta) \to 0$ as $\btheta \to \btheta^*$, i.e., $K(\btheta^*, \cdot)$ is continuous at $\btheta^*$.

To complete the proof, we require Schwartz's theorem \citep{schwartz1965bayes}, which says $\pi(\btheta \in \mathcal{A} \mid \mathcal{D}_N) \to 1$ as $N \rightarrow \infty$ for every open set $\mathcal{A}$ that contains $\btheta^*$ if (i) the parameter space is compact, (ii) the model is identifiable, and (iii) there is positive prior mass on every KL neighborhood of $\btheta^*$, i.e., for all $\epsilon > 0$, $\pi(\mathcal{A}_{\text{KL}}(\btheta^*, \epsilon)) > 0$, where $\mathcal{A}_{\text{KL}}(\btheta^*, \epsilon) := \{ \btheta \in \Theta \mid K(\btheta^*, \btheta) < \epsilon \}$. The first condition is met since $\Theta$ is compact. The second is implied by (A1) and (A2). 
For the third condition, since $K(\btheta^*, \cdot)$ is continuous at $\btheta^*$ and $K(\btheta^*,\btheta^*) = 0$, $\mathcal{A}_{\text{KL}}(\btheta^*, \epsilon)$ contains an open neighborhood of $\btheta^*$. By (A4), the prior therefore assigns positive mass to $\mathcal{A}_{\text{KL}}(\btheta^*, \epsilon)$ for every $\epsilon > 0$. The result follows.\Halmos
\end{proof}

\begin{remark}[Theorem \ref{th:identification} Assumptions]
Assumption (A1) requires that for any parameter $\btheta \neq \btheta^*$, the dataset is sufficiently rich that there exists an instance $\bu^i$ such that the optimal solution with respect to $\btheta$ differs from that with respect to $\btheta^*$. Without (A1), $\Theta^*_N$ may not collapse to a singleton. 
Assumption (A2) requires the observation model be sufficiently informative to distinguish different optimal decisions.  This holds, for example, with a Gaussian noise model such as in Example 1. Assumption (A3) holds for the Gaussian and vMF observation models considered in this paper. In both cases, $\nu(\cdot \mid \bx)$ is continuous in $\bx$, and $\btheta \mapsto \bx^\star(\bu, \btheta)$ is continuous at $\btheta^*$ for almost every $\bu$. Moreover, the corresponding KL divergence is uniformly bounded under compactness of $\mathcal{X(\bu)}$ (Gaussian case) and on the unit sphere (vMF case). Assumption (A4) is standard in the Bayesian inference literature and is satisfied by priors with full support on $\Theta$ such as the uniform distribution.
\end{remark}

\begin{remark}[Relationship to Lemma \ref{lem:set_collapse}]
The assumption in Lemma \ref{lem:set_collapse} strengthens the requirement that each $\btheta \neq \btheta^*$ be distinguishable from $\btheta^*$ on a positive-measure covariate set, by requiring that the distinguishability hold locally uniformly in a neighborhood of $\btheta$. 
This is natural when $\bx^\star(\bu,\cdot)$ is continuous in $\btheta$ and distinct parameter values induce separated optimal solutions on positive-measure sets of covariates, as is typical in well-identified strictly convex forward problems.
Note that this strengthened condition implies (A1) of Theorem~\ref{th:identification}; the additional local uniformity is required only for the set-convergence result and not for posterior consistency.    
\end{remark}

\begin{remark}[Additional parameters]
Theorem~\ref{th:identification} focuses on consistency of $\btheta$. Often, additional parameters are estimated jointly with $\btheta$ (see Section 7). Extending the consistency result in this case follows standard arguments for finite-dimensional parametric models and is omitted for brevity. 
\end{remark}

Theorem~\ref{th:identification} establishes that the posterior concentrates around $\btheta^*$ as the number of observations increases. The proof follows the standard structure of posterior consistency results in Bayesian inference, but is suitably adapted to incorporate the optimization layer in our hierarchical Bayesian model: the identifiability condition (A1) is expressed in terms of covariate variation rather than direct distinguishability of the likelihood, and the KL divergence is integrated over the covariate distribution to account for the role of the optimization in linking parameters to observations.

\section{Uncertainty Quantification}
\label{sec:uq_new}

The posterior distribution of $\btheta$ characterizes the uncertainty in the estimated parameter vector.
Since inference is restricted to the unit hypersphere $\mathcal{S}^{h-1}$, uncertainty is naturally measured in terms of angular (geodesic) distance. To construct geometrically relevant credible regions, we adopt a tangent space approximation based on tools from directional statistics and Riemannian geometry (see, e.g., \cite{mardia2009directional} and \cite{absil2008optimization}). 

\subsection{Constructing Credible Regions}

Assume $\{\btheta^1,\ldots,\btheta^T\}$ is the set of (post warm-up) posterior samples from Algorithm~\ref{algo:decision} or Algorithm~\ref{algo:objective}. 
The construction proceeds in three steps (illustrated in Figure~\ref{fig:mappings}): (a) compute a posterior mean direction, (b) map posterior samples to the tangent space via the logarithmic map, and (c) construct an ellipsoidal credible region in the tangent space. 
This mapping follows standard definitions from Riemannian geometry \citep{absil2008optimization}. 

\subsubsection{Posterior Mean Direction.}

The posterior mean direction is calculated as
\[
\hat{\btheta}
:=
\frac{\frac{1}{T}\sum_{t=1}^T \btheta^t}{\left\|\frac{1}{T}\sum_{t=1}^T \btheta^t\right\|_2}.
\]
It represents the central direction of the posterior distribution.

\begin{figure}[htbp]
  \centering
  \includegraphics[width=0.7\linewidth]{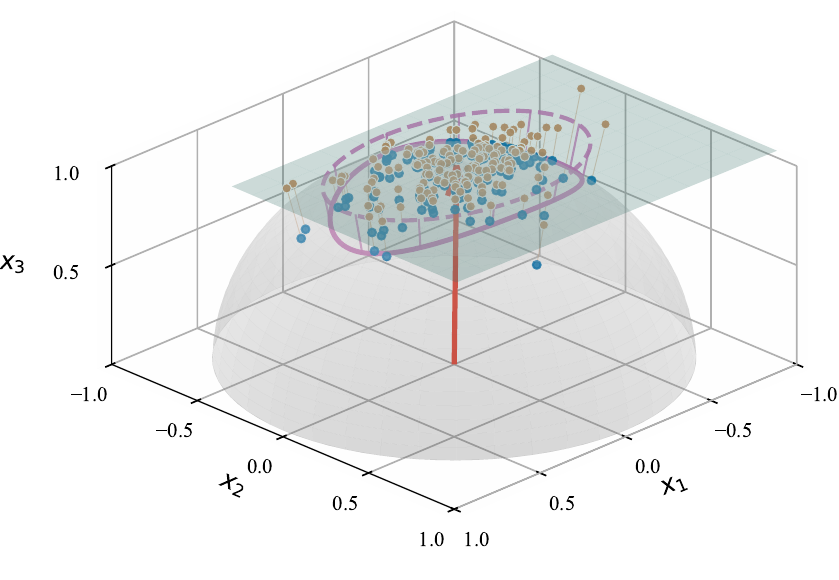}
  \caption{Construction of the credible region on the hypersphere. Posterior samples (blue dots) are mapped to points (orange dots) in the hyperplane tangent to the hypersphere at the mean direction (red arrow). The ellipsoidal region is generated in the tangent space (dashed curve) and can be mapped back to the hypersphere (solid curve).}
  \label{fig:mappings}
\end{figure}

\subsubsection{Mapping to Tangent Space.}

Let
\(
\Delta_{\hat{\btheta}}
=
\{ \bv \in \mathbb{R}^h \,|\, (\bv - \hat{\btheta})^\tpose \hat{\btheta} = 0 \}
\)
denote the affine tangent space to $\mathcal{S}^{h-1}$ at $\hat{\btheta}$. We map each posterior sample $\btheta^{t}$ to a point $\boldeta^t \in \Delta_{\hat{\btheta}}$ via $\boldeta^{t} = \log_{\hat{\btheta}}(\btheta^{t})$, where 
\begin{equation*}
\log_{\hat{\btheta}}(\btheta)
:=
\hat{\btheta}
+
\frac{
\btheta - (\hat{\btheta}^{\tpose} \btheta)\hat{\btheta}
}{
\|\btheta - (\hat{\btheta}^{\tpose} \btheta)\hat{\btheta}\|_2
}
\,
\arccos(\hat{\btheta}^{\tpose} \btheta).
\end{equation*}

The Euclidean distance $\|\boldeta^{t} - \hat{\btheta}\|_2$ in the tangent space equals the geodesic (angular) distance between $\btheta^{t}$ and $\hat{\btheta}$. 
When posterior mass is sufficiently concentrated (as expected under posterior consistency; see Theorem~\ref{th:identification}), geodesic distances between points on the hypersphere will be well-approximated by Euclidean distances of the corresponding points in the tangent space. 

\subsubsection{Ellipsoidal Credible Region.}

Let the empirical mean and covariance matrix of the points in the tangent space be 
\(
\bar{\boldeta}
=
\frac{1}{T}
\sum_{t=1}^{T} \boldeta^{t}
\)
and
\(
\bSigma
=
\frac{1}{T}
\sum_{t=1}^{T}
\left(\boldeta^{t} - \bar{\boldeta}\right)
\left(\boldeta^{t} - \bar{\boldeta}\right)^{\tpose}
\).
Define the squared Mahalanobis distances that measure how far each posterior sample is from the mean direction, accounting for the empirical covariance, by $d^{t} := \left(\boldeta^{t} - \bar{\boldeta}\right)^{\tpose}
\bSigma^{-1}
\left(\boldeta^{t} - \bar{\boldeta}\right)$, $t  = 1, \ldots, T.$

Let $q_{1-\alpha}$ denote the empirical $(1-\alpha)$-quantile of $\{d^1, \ldots d^T\}$. The corresponding ellipsoid  
\begin{equation*}
E_{1-\alpha}
=
\left\{
\boldeta \in \Delta_{\hat{\btheta}}\,\middle|\,
(\boldeta - \bar{\boldeta})^{\tpose}
\bSigma^{-1}
(\boldeta - \bar{\boldeta})
\le q_{1-\alpha}
\right\}
\end{equation*}
summarizes posterior variability in the tangent space. 
The credible region on the hypersphere is therefore 
$
C_{1-\alpha}
=
\left\{
\btheta \in \mathcal{S}^{h-1} \,\middle|\,
\log_{\hat{\btheta}}(\btheta) \in E_{1-\alpha}
\right\}.
$
Geometrically, $C_{1-\alpha}$ forms an anisotropic region on the hypersphere, with shape determined by the posterior covariance structure. Note that $E_{1-\alpha}$ is centered at $\bar{\boldeta}$, which is generally not equal to $\hat{\btheta}$, but this difference vanishes as $N \rightarrow \infty$ and $T \rightarrow \infty$. 

\begin{remark}
Since $\boldeta^t, t = 1, \ldots, T$ lie in the $(h-1)$-dimensional affine subspace $\Delta_{\hat{\btheta}}$, the empirical covariance $\bSigma$ is rank-deficient and its inverse does not exist. We therefore regularize $\bSigma$ by adding a small constant $\epsilon > 0$ along the diagonal, so that it becomes $\bSigma + \epsilon \mathbf{I}$, where $\mathbf{I}$ is the identity matrix. 
For sufficiently small $\epsilon$, this regularization has negligible impact relative to posterior uncertainty in the numerical experiments.
\end{remark}

\subsection{Scalar Summary of Posterior Dispersion}
Next, we propose an interpretable summary of posterior dispersion. Let $\btheta \sim \pi(\cdot \mid \mathcal{D}_N)$ and $\boldeta = \log_{\hat{\btheta}}(\btheta)$, where $\hat{\btheta} := \mathbb{E}[\btheta \mid \mathcal{D}_N] / \|\mathbb{E}[\btheta \mid \mathcal{D}_N]\|_2$. Let $\bSigma_p := \mathrm{Cov}[\boldeta \mid \mathcal{D}_N]$ denote the posterior covariance of $\boldeta$. Let $\mathrm{tr}(\bA)$ denote the trace of matrix $\bA$. We can compactly summarize posterior dispersion using the root mean squared angular deviation of $\btheta$ from the posterior mean direction $\hat{\btheta}$:
\begin{equation*}
\alpha_{\mathrm{RMS}} =
\sqrt{\frac{\mathrm{tr}(\bSigma_p)}{h - 1}}.
\end{equation*}

\begin{proposition}
Under the assumptions of Theorem~\ref{th:identification}, $\alpha_{\mathrm{RMS}} \to 0$ as $N \to \infty$.
\end{proposition}

\begin{proof}{Proof.}
Since $\boldeta = \log_{\hat{\btheta}}(\btheta)$, $\|\boldeta - \hat{\btheta}\|_2 = \arccos(\hat{\btheta}^\top \btheta)$. By Theorem~\ref{th:identification}, $\btheta \to \btheta^*$ in probability as $N \to \infty$, so $\hat{\btheta} \to \btheta^*$ and  $\|\boldeta - \hat{\btheta}\|_2 \rightarrow 0$ in probability as $N \rightarrow \infty$. Since $\arccos(\cdot) \le \pi$, $\|\boldeta - \hat{\btheta}\|_2$ is uniformly bounded above for all $\btheta \in \mathcal{S}^{h-1}$ and all $N$. This implies uniform integrability, and so convergence in probability implies convergence of second moments:
$\mathbb{E}\big[\|\boldeta - \hat{\btheta}\|_2^2 \,\big|\, \mathcal{D}_N\big] \to 0$.

The trace of $\bSigma_p$ equals the sum of the variances along the diagonal,
\begin{equation*}
\begin{aligned}
  \mathrm{tr}(\bSigma_p)
  &= \mathbb{E}\big[\|\boldeta - \mathbb{E}[\boldeta \mid \mathcal{D}_N]\|_2^2 \,\big|\, \mathcal{D}_N\big] \\
  &\le \mathbb{E}\big[\|\boldeta - \hat{\btheta}\|_2^2 \,\big|\, \mathcal{D}_N\big],
\end{aligned}
\end{equation*}
where the inequality follows from the fact that the quantity that minimizes the mean squared distance from $\boldeta$ is its expectation. Hence, $\mathrm{tr}(\bSigma_p) \to 0$ as $N \to \infty$, implying $\alpha_{\mathrm{RMS}} \to 0$.
\Halmos
\end{proof}

In practice, $\alpha_{\mathrm{RMS}}$ is computed using the empirical covariance of MCMC samples in place of the posterior covariance $\bSigma_p$. 
This empirical covariance converges almost surely to $\bSigma_p$ as $T \to \infty$ \citep{robert2004monte}, so the empirical $\alpha_{\mathrm{RMS}}$ inherits the same asymptotic behavior as $\alpha_{\mathrm{RMS}}$ defined by $\bSigma_p$. We report the empirical $\alpha_{\mathrm{RMS}}$ in our numerical experiments in Section~\ref{sec:numericals}.

\section{Numerical Experiments} 
\label{sec:numericals}

We illustrate the application of our algorithms to estimate credible regions in linearly and quadratically constrained problems with a linear objective. To distinguish them, we will refer to them as LP and QP, respectively. The focus is on empirically assessing coverage of the credible regions. 

\subsection{Simulation Setup}

Details on the optimization model formulations and the corresponding simulated datasets are given in \ref{ec:numerical}. The data-generating and inference procedures are repeated 100 times. Coverage is defined as the proportion of those simulations in which the true parameter vector $\bthetastar$ lies within the constructed credible region. All credible regions are constructed at the $95\%$ ($\alpha = 0.05$) level following the approach from Section~\ref{sec:uq_new}. 

In Algorithm~\ref{algo:decision}, we assume the likelihood $f$ is Gaussian with mean $\bx^\star(\bu,\btheta^*)$ and unknown covariance matrix $\sigma^2\mathbf{I}$. 
We place a uniform prior on $\btheta$ and a half-Cauchy prior on $\sigma$ \citep{gelman2013bayesian}. We update $\btheta$ and $\sigma$ within each MCMC iteration separately using Gaussian random walks. 
In Algorithm~\ref{algo:objective}, we assume a von Mises-Fisher perturbation distribution $p$ centered at $\btheta^*$ with unknown concentration parameter $\kappa$. We place a uniform prior on $\btheta$ and a Gamma prior on $\kappa$, 
and update them using a Gaussian random walk for $\btheta$ and a log-normal proposal for $\kappa$. 
The IO problems are solved as outlined in \ref{sec:io_solves}. 

For both algorithms, we choose a Gaussian proposal distribution $q$ because it is easy to sample from in any dimension and the covariance matrix can be tuned to match the shape and scale of the posterior distribution, which is important for efficient mixing in high dimensions. Additional details regarding the proposal distribution including alternatives considered are provided in \ref{sec:proposal}.

We run each chain until the PSRF falls below 1.1 across all parameters; in our experiments, this required between 10{,}000 and 100{,}000 iterations per chain. A warm-up period of 25\% of the total iterations is discarded in all cases. 
The reported times are the wall-clock durations of one run of each algorithm on a single processor, including the warm-up iterations.

\subsection{Results}

Table~\ref{tab:lp} summarizes the empirical coverage percentages and the mean $\alpha_\text{RMS}$ ($\pm$ 1 standard deviation) across 100 replications, obtained by applying each algorithm to an LP model with data generated under its corresponding error model.
We vary the number of observed decisions ($N$), the variance  ($\sigma^2$) of the Gaussian distribution in Algorithm \ref{algo:decision}, and the concentration parameter ($\kappa$) of the vMF distribution (see~\ref{ec:vmf}) in Algorithm \ref{algo:objective}. Table~\ref{tab:qp} reports the analogous results for the QP.

Both tables show that empirical coverage is near the nominal 95\% level for all $N$, $\sigma$ and $\kappa$. Both the mean and standard deviation of $\alpha_\text{RMS}$ are largest for small $N$. These values are particularly large in the LP case because of the relatively few data points to resolve the nonidentifiability. In all cases, the size and variability of the credible region decreases with increasing $N$, as expected. 

The credible regions for the QP tend to be smaller, particularly for small $N$, reflecting the strict convexity of the QP feasible region, which ensures that each observation provides 
precise identifying information about the true parameter vector. In the LP case, since multiple parameter vectors can generate the same optimal solution, residual directional uncertainty manifests through larger credible regions. 
Second, the standard deviation of $\alpha_\text{RMS}$ is markedly smaller and more stable across 
replications for the QP than for the LP, particularly under Algorithm~\ref{algo:decision}. 
Since the QP is identifiable, the posterior concentrates consistently across replications, yielding stable credible region sizes. For the LP, the directions that remain uncertain vary across replications depending on the realized covariate set, inflating variability. 

Finally, we find that Algorithm \ref{algo:decision} is more computationally demanding than Algorithm \ref{algo:objective}, as expected. The computation times for Algorithm \ref{algo:decision} grow approximately linearly with $N$, whereas for Algorithm \ref{algo:objective}, they grow much more slowly with
$N$ (particularly for the QP). The result is about two orders of magnitude less time for the largest $N$ values we considered.

\begin{table}[t]
\centering
\small
\caption{Applying Algorithms~\ref{algo:decision} and~\ref{algo:objective} to LP instances.}
\label{tab:lp}
\begin{tabular}{cccccccccc}
\toprule
 & \multicolumn{4}{c}{Algorithm 1} & \multicolumn{4}{c}{Algorithm 2 } \\
\cmidrule(r){2-5} \cmidrule(l){6-9}
$N$ & $\sigma^2$ & Coverage (\%) & $\alpha_{\mathrm{RMS}}$ (deg) & Time (s)
    & $\kappa$ & Coverage (\%) & $\alpha_{\mathrm{RMS}}$ (deg) & Time (s) \\
\midrule

50   & 0.05 & 95 & 42.61 $\pm$ 13.42 & 1703 & 10 & 97 & 39.96 $\pm$ 10.92 & 161 \\
100  &      & 96 & 26.40 $\pm$ 9.16  & 3771 &    & 99 & 26.36 $\pm$ 4.47  & 168 \\
200  &      & 94 & 14.29 $\pm$ 4.12  & 8542 &    & 97 & 17.82 $\pm$ 2.01  & 330\\
500  &      & 97 & 5.95 $\pm$ 2.04   & 36,941 &    & 93 & 10.88 $\pm$ 0.83  & 449\\
1000 &      &  95  & 3.34 $\pm$ 1.48                 &   81,281 &  & 94 & 7.75 $\pm$ 0.37   & 653\\

\midrule

50   & 0.1 & 94 & 54.16 $\pm$ 16.29 & 1517 & 5  & 96 & 52.16 $\pm$ 17.58 & 164 \\
100  &     & 96 & 32.47 $\pm$ 10.82 & 3627 &    & 95 & 33.05 $\pm$ 8.50  & 166 \\
200  &     & 93 & 19.01 $\pm$ 5.95  & 8044 &    & 94 & 22.71 $\pm$ 3.78  & 230 \\
500  &     & 96 & 8.58 $\pm$ 2.67   & 37,349 &    & 97 & 13.94 $\pm$ 1.25  & 368 \\
1000 &     & 94 & 4.75 $\pm$ 1.42   & 39,399 &    & 98 & 9.99 $\pm$ 0.72   & 554\\

\bottomrule
\end{tabular}
\end{table}

\begin{table}[t]
\small
\centering
\caption{Applying Algorithms~\ref{algo:decision} and~\ref{algo:objective} to QP instances.}
\label{tab:qp}
\begin{tabular}{cccccccccc}
\toprule
 & \multicolumn{4}{c}{Algorithm 1} & \multicolumn{4}{c}{Algorithm 2} \\
\cmidrule(r){2-5} \cmidrule(l){6-9}
$N$ & $\sigma^2$ & Coverage (\%) & $\alpha_{\mathrm{RMS}}$ (deg) & Time (s)
    & $\kappa$ & Coverage (\%) & $\alpha_{\mathrm{RMS}}$ (deg) & Time (s) \\
\midrule

50   & 0.05 & 97  & 11.01 $\pm$ 0.31  & 499 & 10 & 97  & 21.59 $\pm$ 1.60 & 234 \\
100  &      & 93  & 8.60 $\pm$ 0.31   & 903 &    & 96  & 15.55 $\pm$ 0.89 & 254 \\
200  &      & 95  & 6.40 $\pm$ 0.27   & 2559 &    & 95  & 11.13 $\pm$ 0.37 & 278 \\
500  &      & 95  & 5.24 $\pm$ 0.27   & 6604 &    & 97  & 7.34 $\pm$ 0.18  & 307 \\
1000 &      & 96  & 3.67 $\pm$ 0.25   & 23,435 &    & 97  & 5.61 $\pm$ 0.13  & 355 \\

\midrule

50   & 0.1  & 95  & 14.58 $\pm$ 0.49  & 487 & 5  & 94  & 39.94 $\pm$ 5.67 & 206 \\
100  &      & 95  & 10.88 $\pm$ 0.32  & 863 &    & 97  & 28.25 $\pm$ 2.79 & 220 \\
200  &      & 97  & 8.01 $\pm$ 0.25   & 2460 &    & 94  & 20.24 $\pm$ 1.40 & 244\\
500  &      & 95  & 6.07 $\pm$ 0.25   & 6450 &    & 96  & 13.04 $\pm$ 0.50 & 299\\
1000 &      & 97  & 4.73 $\pm$ 0.22   & 20,102 &    & 96  & 9.42 $\pm$ 0.27 & 329\\

\bottomrule
\end{tabular}
\end{table}

\section{Conclusions}

This paper develops a probabilistic modeling framework for inverse optimization that enables uncertainty quantification of a vector that parameterizes a linear cost function. Using Bayesian inference, we construct credible regions that are both geometrically consistent and computationally tractable. Under appropriate identifiability conditions, we show that the posterior concentrates around the true parameter. We further demonstrate that the resulting credible regions achieve near-nominal empirical coverage. This framework provides a principled way to propagate uncertainty into downstream optimization models or calibrate a decision maker's confidence in IO-based inferences.

\bibliographystyle{informs2014}
\bibliography{references}

\renewcommand{\notesname}{Endnotes}   
\theendnotes

\newpage

\ECHead{Electronic Companion}

\setcounter{section}{0}
\renewcommand{\thesection}{EC.\arabic{section}}
\renewcommand{\theequation}{EC.\arabic{equation}}
\setcounter{equation}{0}
\renewcommand{\thefigure}{EC.\arabic{figure}}
\setcounter{figure}{0}

\section{von Mises--Fisher Distribution}
\label{ec:vmf}

The von Mises--Fisher (vMF) distribution is a probability distribution defined on the unit hypersphere 
$\mathcal{S}^{h-1} = \{\btheta \in \mathbb{R}^h \mid \|\btheta\|_2 = 1\}$ and is commonly used to model directional data \citep{mardia2009directional}. It can be viewed as the analogue of the multivariate normal distribution for unit-norm vectors.

A random vector $\tilde{\btheta} \in \mathcal{S}^{h-1}$ follows a vMF distribution with mean direction 
$\btheta \in \mathcal{S}^{h-1}$ and concentration parameter $\kappa \ge 0$, denoted by 
$\tilde{\btheta} \sim \mathrm{vMF}(\btheta, \kappa)$, if its density is given by
\[
p(\tilde{\btheta} \mid \btheta, \kappa) 
= C_h(\kappa) \exp\left( \kappa \, \btheta^\tpose \tilde{\btheta} \right),
\]
where $C_h(\kappa)$ is the normalizing constant
\[
C_h(\kappa) = \frac{\kappa^{h/2 - 1}}{(2\pi)^{h/2} I_{h/2 - 1}(\kappa)},
\]
and $I_\nu(\cdot)$ denotes the modified Bessel function of the first kind of order $\nu$.

The parameter $\btheta$ represents the mean direction of the distribution, while $\kappa$ controls the concentration of the distribution around $\btheta$. When $\kappa = 0$, the distribution reduces to the uniform distribution over $\mathcal{S}^{h-1}$. As $\kappa \to \infty$, the distribution becomes increasingly concentrated around $\btheta$.

In our experiments, the vMF distribution is used to model perturbations of the parameter vector in the objective-space error setting, as it naturally respects the unit-norm constraint on $\btheta$.

\section{Posterior Visualizations for Examples 1 and 2}
\label{ec:figures}

\begin{figure}[h]
    \centering
    \begin{subfigure}[b]{0.48\textwidth}
        \centering
        \includegraphics[trim=.2in .5in .35in 0in, clip, width=.49\textwidth]{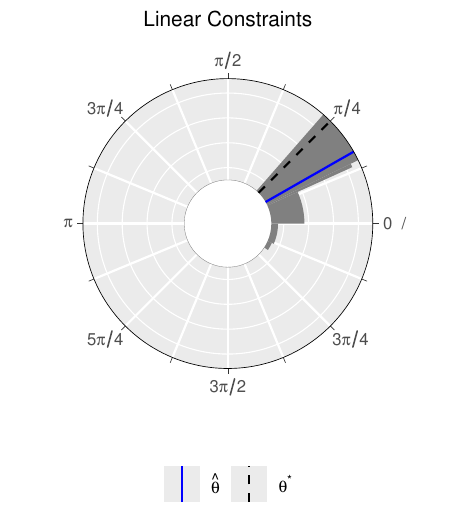}
        \includegraphics[trim=.2in .5in .35in 0in, clip, width=.49\textwidth]{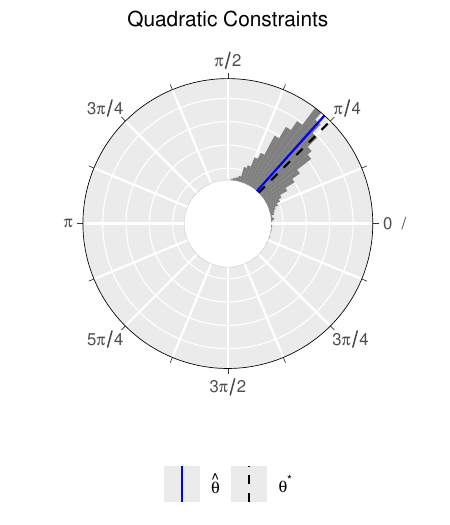}
        \caption{Example 1}
        \label{fig:posterior_dec}
    \end{subfigure}
     \hspace{0em} \vrule
    \begin{subfigure}[b]{0.48\textwidth}
        \centering
        \includegraphics[trim=.2in .5in .35in 0in, clip, width=.49\textwidth]{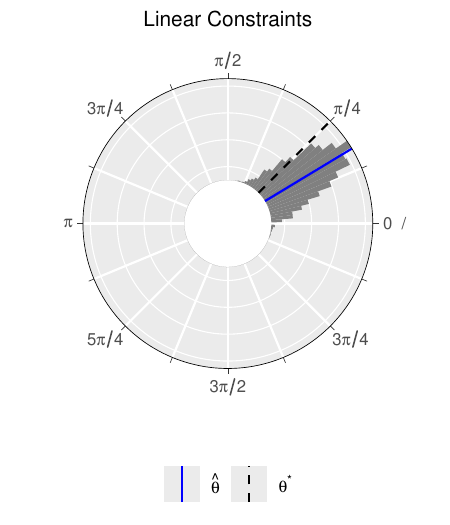}
        \includegraphics[trim=.2in .5in .35in 0in, clip, width=.49\textwidth]{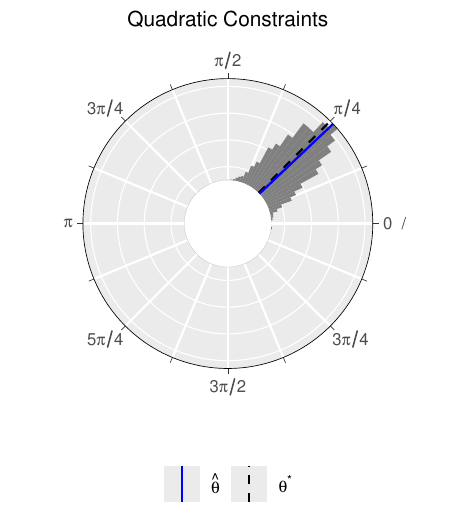}
        \caption{Example 2}
        \label{fig:posterior_obj}
    \end{subfigure}
\caption{Histograms of 10{,}000 posterior samples (post warm-up) of the angular coordinate of $\btheta$. Panel (a) corresponds to Example 1 (decision-space error) and panel (b) to Example 2 (objective-space error). In each plot, the dashed black line indicates the direction of the true parameter $\btheta^*$ and the solid blue line indicates the posterior mean direction $\hat{\btheta}$.}
    \label{fig:posterior}
\end{figure}

Figure~\ref{fig:posterior} shows histograms of 10{,}000 MCMC posterior samples (post warm-up) of the angular coordinate of $\btheta$ drawn from Algorithms \ref{algo:decision} and \ref{algo:objective} applied to the data in Examples \ref{example:decspace} and \ref{example:objspace}, respectively. In each case, we use a von Mises–Fisher distribution ($\kappa = 5$) for the proposal distribution $q(\cdot\mid\btheta)$. 

In Example 1 (decision-space error) with linear constraints, the step-like structure of the posterior arises from the many-to-one mapping between $\btheta$ and $\bx^\star(\bu^i,\btheta)$. This means the likelihood of $\bx^\star(\bu^i,\btheta)$ depends on $\btheta$ only through the ``arc'' in which it lies. 
Since all $\btheta$ in the same arc share an identical likelihood and the prior is uniform, the posterior is constant within each arc and jumps only at the arc boundaries.
In contrast, in Example 1 with quadratic constraints, there is a one-to-one mapping between $\btheta$ and $\bx^\star(\bu^i,\btheta)$, so the likelihood varies smoothly with $\btheta$ and the posterior exhibits a continuous shape. 

Interestingly, Example 2 (objective-space error) with linear constraints does not yield a step-like posterior. This is because the likelihood 
$\nu(\by^i \mid \bx^\star(\bu^i, \btheta)) = \int_{\left\{ \tilde{\btheta} \in \Theta \;\middle|\; 
\bx^\star(\bu,\tilde{\btheta}) = \by \right\}} p(\tilde{\btheta} \mid \btheta)\, d\tilde{\btheta}$ 
(Equation \eqref{eq:obj_err_likelihood} in Section \ref{sec:identifiability}) depends on $\btheta$ through the integral of a smooth perturbation density (vMF) over the fixed arc ${\left\{ \tilde{\btheta} \in \Theta \;\middle|\; 
\bx^\star(\bu,\tilde{\btheta}) = \by \right\}}$. Thus, when there is a many-to-one mapping from $\btheta$ to $\bx^\star(\bu^i, \btheta)$, the likelihood varies smoothly with $\btheta$, producing a continuous posterior. 
With quadratic constraints, the one-to-one mapping between $\btheta$ and $\bx^\star(\bu^i, \btheta)$ implies that the likelihood is determined by the unique perturbed cost recovered from $\by^i$ via \eqref{eq:inverse_prob_alg2}, i.e., $\nu(\by^i \mid \bx^\star(\bu^i, \btheta)) = p(\tilde{\btheta}^i \mid \btheta)$, where $\tilde{\btheta}^i$ is the unique perturbed cost associated with $\by^i$, and is thus shaped by the smooth vMF likelihood.

\section{Numerical Experiment Details}
\label{ec:numerical}

\subsection{Models}

In our experiments, we assume $h=n$. Let $\bx \in \R^n$ be the decision vector. Let $\bA \in \R^{m \times n}$, $\bb \in \R^m$, $\bQ  \in \R^{n\times n}$ be a positive definite matrix, and $\beee$ be the vector of all ones. The two models we consider are
\begin{equation*} 
\begin{aligned}
         \mathbf{LP}: \quad \underset{\bx}{\text{minimize}} \quad & \btheta^{\tpose} \bx \nonumber \\
        \text{subject to} \quad & \bA\bx \ge \bb, \\ & -\beee\leq \bx \leq \beee
\end{aligned}
\end{equation*} 
and 
\begin{equation*} 
\begin{aligned}
       \mathbf{QP}: \quad \underset{\bx}{\text{minimize}}\quad & \btheta^{\tpose} \bx \nonumber \\
        \text{subject to}  \quad & \bx^{\tpose} \bQ \bx +\beee^{\tpose}\bx \leq 1.
\end{aligned}
\end{equation*} 
The box constraints are included to ensure the LP has a finite optimum.

For the LP, we set $(m,n) = (75,20)$. 
The rows of $ \bA \in \mathbb{R}^{m \times n} $ are generated by sampling $ \tilde{\ba}_j \sim \mathcal{N}(0, \mathbf{I}_n) $, normalizing, and scaling by independent factors $ s_j \sim \mathrm{Unif}(0.7, 1.4) $, i.e.,
\[
\ba_j := s_j \frac{\tilde{\ba}_j}{\|\tilde{\ba}_j\|_2}, \quad j = 1, \dots, m.
\]
We then sample $ \bw \sim \mathrm{Unif}([-0.72, 0.72]^n)$ and define
$
\bb = \bA \bw - \boldsymbol{\delta},
$
where $\boldsymbol{\delta} \sim \mathrm{Unif}([0.04, 0.28]^m)$. This ensures strict feasibility of $\bw$ for the constraints $\bA\bx \ge \bb$. The parameter ranges are chosen to yield well-conditioned instances; we verified that the results are robust to moderate variations in these ranges.

For the QP, we set $n = 20$ and generate matrices $\bQ$ with eigenvalues uniformly distributed between 1 and 5.

\subsection{Datasets}

For each formulation, a dataset $\mathcal{D} = \{(\bu^i, \by^i)\}_{i=1}^N$ was generated comprising $N$ instances. For instance $i$, the covariate $\bu^i$ consists of $(\bA^i, \bb^i)$ for the LP and $\bQ^i$ for the QP.
We set the true parameter vector to $\bthetastar = \beee/\sqrt{n}$.
To generate $\by^i$ in the decision error model, we solve $N$ instances of the forward problem to obtain optimal solutions $\bx^\star(\bu^i,\btheta^*)$, $i = 1, \ldots, N$, and generate observations as
\[
\by^i = \bx^\star(\bu^i,\btheta^*) + \bepsilon_i,
\]
where $\bepsilon_i \sim \mathcal{N}(0,\sigma^2 \mathbf{I})$ are independent across instances. For the objective error model, we draw perturbed parameter vectors independently from a von Mises--Fisher distribution on $\mathcal{S}^{n-1}$ with mean direction $\bthetastar$ and concentration parameter $\kappa$, and solve the forward problems to optimality to obtain $\by^i$, $i = 1, \ldots, N$.

\subsection{Solving the Inverse Optimization Model in Algorithm~\ref{algo:objective}}
\label{sec:io_solves}

For the LP, we determine active constraints at the observed $(\bu^i,\by^i)$ and choose $\tilde{\btheta}^i$ as a normalized vector in the corresponding normal cone.
For the QP, we set $\tilde{\btheta}^i = (-2 \bQ^i \by^i - \beee)/(\|-2 \bQ^i \by^i - \beee\|_2)$, which is the normalized vector orthogonal to the feasible region defined by $(\bu^i, \by^i)$.

\subsection{Proposal Distribution}
\label{sec:proposal}

To generate $\btheta^c$, we sample $\bz \sim \mathcal{N}(\btheta^{t-1}, \mathbf{S})$ and project to the hypersphere via $\btheta^c \leftarrow \bz/\|\bz\|_2$. Because the Gaussian distribution has support on $\mathbb{R}^n$ rather than $\mathcal{S}^{n-1}$, this approach induces a proposal density that is not exactly symmetric, so the acceptance probability should include the proposal ratio $q(\btheta^{t-1} \mid \btheta^c) / q(\btheta^c \mid \btheta^{t-1})$. In our implementation, however, we adopt a symmetric approximation for numerical efficiency. This is justified because when the largest eigenvalue of $\mathbf{S}$ is small (i.e., $\ll 1$), proposals remain in a small neighborhood of $\btheta^{t-1}$ on $\mathcal{S}^{n-1}$ and the sphere is approximately flat locally, so the proposal ratio is close to 1 and the resulting bias is negligible. Under posterior consistency (established in Theorem~\ref{th:identification}), the posterior concentrates around $\btheta^*$ as $N \to \infty$, implying that well-calibrated proposals will have small step sizes. The near-nominal empirical coverage observed in our numerical experiments supports the validity of this approximation. 

To improve sampling efficiency, we adopt the adaptive Metropolis scheme of \citet{haario2001adaptive}, which updates the proposal covariance matrix $\mathbf{S}$ during the warm-up phase to target an acceptance rate between 25\% and 40\%.

We experimented with two alternative symmetric proposal distributions: (i) a Gaussian distribution with mean $\btheta^{t-1}$ and isotropic covariance $\sigma^2_p \mathbf{I}$, with $\sigma^2_p$ adapted during warm-up, and (ii) a vMF distribution with mean direction $\btheta^{t-1}$ and concentration parameter $\kappa_p$, also adapted during warm-up. We compared these to the Gaussian proposal with the full covariance matrix $\mathbf{S}$ across LP and QP instances of dimension $n$ ranging from 2 to 20. In low dimensions ($n \le 5$), all three proposals yielded comparable posterior summaries ($\alpha_\text{RMS}$ and coverage) and similar sampling efficiency. In higher dimensions, the full-covariance Gaussian proposal substantially outperformed both alternatives in efficiency, requiring substantially fewer iterations (by an order of magnitude), to reach comparable PSRF values. This reflects its ability to adapt to the shape of the posterior through $n(n+1)/2$ covariance parameters, rather than the single scalar of the isotropic Gaussian or vMF proposals.

\section{Code and Data}
Code and data can be found at \url{https://github.com/nyousefi2020/UQ_in_IO}.
All experiments were run on a MacBook Pro (2021) with an Apple M1 Max processor and 64 GB of memory. Optimization problems were solved using Gurobi \citep{gurobi}.







\end{document}